\documentclass[12pt]{article}
\usepackage{amsmath,amssymb,amsthm}
\usepackage{cite}
\usepackage{enumitem}
\usepackage{authblk}
\usepackage{blindtext}
\usepackage[pagebackref=false,colorlinks,linkcolor=blue,citecolor=magenta]{hyperref}
\numberwithin{equation}{section}

\newtheorem{theorem}{Theorem}[section]

\newtheorem{definition}{Definition}[section]
\newtheorem{remark}{Remark}[section]
\begin{document}
%\begin{frontmatter}
\title{Multiple weak solutions for a kind of time-dependent equation involving singularity}

\author{F. Abdolrazaghi\thanks{f.abdolrazaghi@edu.ikiu.ac.ir}}%fatemeh\_abdolrazaghi@yahoo.com

%\thanks{The first author was supported in part by the XXX Fund, Grant No.~YYYYY.}
\author{A. Razani\thanks{Corresponding author}\thanks{razani@sci.ikiu.ac.ir}}
\author{R. Mirzaei\thanks{r.mirzaei@sci.ikiu.ac.ir}}
\affil{Imam Khomeini International University\\ Faculty of  Science, Department of Pure Mathematics\\ Postal code: 34149-16818,
  Qazvin, Iran}

%\corresponding{author}
\date{\empty}
\maketitle
\begin{abstract}
The existence of at least three weak solutions for a kind of nonlinear time-dependent
equation is studied. In fact, we consider the case that the source function has singularity at origin.
To this aim, the variational methods and the well-known critical points theorem are main tools.
\end{abstract}



%\end{frontmatter}

\section{Introduction}
The linear Sobolev equations have a real physical background (\cite{barenblatt1960basic,shi1990initial,ting1974cooling}) and are studied in \cite{davis1972quasilinear,ewing1977coupled}. Because of their complexity,
%of the domain and the source functions
 they haven't exact solutions (except some very especial cases \cite{aristov2017exact}). There are different methods to study the solution of these problems. One of the standard methods is the fixed point theory that investigate the existence of solutions of nonlinear boundary
value problems \cite{agarwal2010survey,benchohra2009boundary,zhang2010positive,Mokhtarzadeh12,EhsaniAA,
PournakiAML,GoodarziAAA,Razani2014,dinmohammadi2017analytical,dinmohammadi2017existence}. The calculus of variation is another impressive technique and for using this technique, one needs to show that the given boundary value problem should possess a variational structure on some convenient spaces
\cite{AbdolrazaghiMMN,BehboudiFilmat,chu2017weak,corvellec2010doubly, heydari2018efficient,JeanMawhin,
Khalkhali2013,Khalkhali2012,MakvandFilomat, MahdaviFilomat, MakvandMJM, MakvandCKMS, MakvandTJM,MakvandGMJ,li2005existence,nieto2009variational,rabinowitz1986minimax,ragusa2004cauchy,tang2010some,goudarzi2019weak}.

In the present paper, we study the weak solutions of
\begin{equation}\label{In1}
\left\{
\begin{array}{l}
\frac{\partial u}{\partial t}-\frac{\partial(\triangle u)}{\partial t}=\mu f(x,t,u) \ \  in\  \Omega,\\ %\tage{In1}\\
u=0\ \ on\ \partial\Omega,\\% \tage{In2}\\
u(x,0)=g(x)  \ \ x\in\Omega,\\%  \tage{In3}
\end{array}
\right.
\end{equation}
where
$\Omega$ is a non-empty bounded open subset of $\mathbb{R}^N$
with  $\partial \Omega\in C^1$, $\mu$ is a positive parameter,
$f:\Omega\times\mathbb{R}^{+}\times\mathbb{R}\rightarrow\mathbb{R} $ is a Carath\'{e}odory function and
has a singularity at the origin with respect to the time variable and $g:\Omega\rightarrow \mathbb{R}$ vanishes on $\partial \Omega$.\\
The aim of this paper is to find an interval for $\mu$ for which the problem
\eqref{In1} admits at least three distinct weak solutions.

By integrating the first equation of \eqref{In1} we get
\begin{equation}\label{Co1}
\int_{0}^{t}\frac{\partial u(x,s)}{\partial s}ds-\int_{0}^{t}\frac{\partial\Delta u(x,s)}{\partial s}ds=\int_{0}^{t}\mu f(x,s,u)ds,
\end{equation}
or
\begin{equation}\label{Co001}
-\Delta u(x,t)=\mu F(x,t,u)-u(x,t)+g(x)-\Delta g(x),
\end{equation}
where
\begin{equation}\label{Co3}
F(x,t,u)=\int_{0}^{t}f(x,s,u)ds.
\end{equation}
The equation \eqref{Co001} is a time-dependent elliptic equation.

\begin{definition}
A function $u:\Omega\rightarrow \mathbb{R}$
is called a weak solution of
the problem
\eqref{In1}
if
$u\in H_{0}^{1}$
and
\begin{equation}\label{p2}
\begin{array}{rl}
\int_{\Omega}\nabla u(x,t)\cdot\nabla v(x)dx&-\mu\int_{\Omega}F(x,t,u(x))v(x)dx+\int_{\Omega}u(x,t)v(x)dx\\
&-\int_{\Omega}g(x)v(x)dx+\int_{\Omega}\Delta g(x)v(x)dx=0,
\end{array}
\end{equation}
for all
$v\in H_{0}^{1}$ and $t\geq0$.
\end{definition}
\begin{definition}\label{functional}
Define the functionals
$\varphi,\vartheta:H_{0}^{1}\rightarrow \mathbb{R}$
 by $\varphi(u):=\frac{1}{2}{\|u\|}^{2}$
and
\[
\begin{array}{rl}
\vartheta(u):=&\int_{\Omega}\widetilde{F}(x,t,u)dx-\frac{1}{2\mu}\int_{\Omega}\left(u(x,t)
\right)^2dx+\frac{1}{\mu}\int_{\Omega}g(x)u(x,t)dx\\ & \\
& \quad \quad -\frac{1}{\mu}\int_{\Omega}\Delta g(x)u(x,t)dx,
\end{array}
\]
respective, where $\widetilde{F}(x,t,\eta):=\int_{0}^{\eta}F(x,t,s)ds$.
\end{definition}
Notice that $\varphi$ and $\vartheta$ are well-defined and $ C^1$, $\varphi^\prime,\vartheta^\prime\in X^{*}$,
$\varphi^\prime(u)(v)=\int_{\Omega}\nabla u(x)\cdot\nabla v(x)dx$
and
\[
\begin{array}{rl}
\vartheta^\prime(u)(v)=&\int_{\Omega}F(x,t,u(x))v(x)dx-\frac{1}{\mu}
\int_{\Omega}u(x,t)v(x)dx\\
& \\
&\quad \quad +\frac{1}{\mu}\int_{\Omega}g(x)v(x)dx-\frac{1}{\mu}\int_{\Omega}\Delta g(x)v(x)dx.
\end{array}
\]
\begin{remark}
A critical point of  $I_\mu:=\varphi-\mu\vartheta$ is exactly a weak solution of
\eqref{In1}.
\end{remark}
%The Sobolev Embedding Theorem
%\cite{bonanno2011three}, says $H_0^1(\Omega)\hookrightarrow L^{2^*}(\Omega)$
% i.e. there exists
%$c\in \mathbb{R}^+$ such that for all $u \in H_0^1(\Omega)$
%\begin{equation}\label{p3}
%\|u\|_{ L^{2^*}(\Omega)}\leq c\|u\|,
%\end{equation}
%where
%\begin{equation}\label{p4}
%c:=\frac{1}{\sqrt{N(N-2)\pi}}\left(\frac{N!}{2\Gamma(N/2+1)}\right)^{\frac{1}{N}},
%\end{equation}
%$\Gamma(z)$ is the Gamma function and $2^*=2N/(N-2)$.
Fix $q\in [1,2^*[$, Embedding Theorem \cite{bonanno2011three} shows
$H_0^1(\Omega)\overset{c}\hookrightarrow L^{q}(\Omega)$, i.e.
there exists $c_q>0$ such that for all $u \in H_0^1(\Omega)$
\begin{equation}\label{p6}
\|u\|_{ L^{q}(\Omega)}\leq c_q\|u\|,
\end{equation}
where
\begin{equation}\label{p7}
c_q\leq\frac{meas(\Omega)^{\frac{2^*-q}{2^*q}}}{\sqrt{N(N-2)\pi}}\left(\frac{N!}{2\Gamma(N/2+1)}\right)^{\frac{1}{N}},
\end{equation}
$\Gamma$ is the Gamma function, $2^*=2N/(N-2)$
and $meas(\Omega)$ denotes the Lebesgue measure of $\Omega$.

%%%%%%%%%%%%%%%%%%%%%%%%%%%%%%%%%%%%%%%%%%%%%%%%%%%%%%%%%%%%%%%%%%%%%%%%%%%%%%%%%%%%%%%%%%%%%%%%%%%%%%%%%%%%%%%%%%%%
\section{Three weak solutions}
In this section the existence of at least three weak solutions for the problem \eqref{In1} is proved. Due to do this, we apply \cite[Theorem 3.6]{bonanno2010structure} which is given below
\begin{theorem}\label{p1}(see \cite{bonanno2010structure}, Theorem 3.6).
let
$X$
be a reflexive real Banach space,
$\Phi:X\rightarrow \mathbb{R}$
be a coercive, continuously Gateaux differentiable and sequentially weakly lower semicontinuous functional whose Gateaux derivative admits a continuous inverse on
$X^{*}$,
$\Psi:X\rightarrow \mathbb{R}$
be a continuously Gateaux differentiable functional whose Gateaux derivative is compact such that
$\Phi(0)=\Psi(0)=0.$
Assume that there exist
$r>0$
and
$\overline{x}\in X$,
with
$r<\Phi(\overline{x})$,
such that:
\begin{enumerate}
\item
$\frac{\sup_{\Phi(x)\leq r}\Psi(x)}{r}<\frac{\Psi(\overline{x})}{\Phi(\overline{x})};$
\item
for each
$\lambda\in\Lambda_{r}:=]\frac{\Phi(\overline{x})}{\Psi(\overline{x})},\frac{r}{\sup_{\Phi(x)\leq r}\Psi(x)}[$
the functional
$\Phi-\lambda\Psi$
is coercive.
\end{enumerate}
Then, for each
$\lambda \in \Lambda_{r}$,
the functional
$\Phi-\lambda\Psi$
has at least three distinct critical points in
$X$.
\end{theorem}
Set
\begin{equation}\label{p8}
\begin{array}{ll}
D:=\underset{x\in\Omega}\sup \ dist(x,\partial\Omega),&
\kappa:=\frac{D\sqrt{2}}{2{\pi}^{N/4}}\left(\frac{\Gamma(N/2+1)}{D^N-(D/2)^N}\right)^{\frac{1}{2}},\\
K_1:=\frac{2\sqrt{2}c_1(2^N-1)}{D^2},&
K_2:=\frac{2^{\frac{q+2}{2}}c_q^q(2^N-1)}{qD^2}.
\end{array}
\end{equation}
Now, we can state the main result.

\begin{theorem}\label{Maint}
Let
$f:\Omega\times\mathbb{R}^{+}\times\mathbb{R}\rightarrow\mathbb{R}$
be a Carath\'{e}odory function and $g:\Omega\rightarrow \mathbb{R}$ vanishes on $\partial \Omega$. Assume
\begin{itemize}
\item[(1)] There exist non-negative constants $m_1$,$m_2$ and $q \in ]1,\frac{2N}{N-2}[$ such that
$$F(x,t,s)\leq m_1+m_2\mid s\mid^{q-1}+\frac{1}{\mu}\left(s-g(x)+\Delta g(x)\right)$$
for all $(x,t,s)\in\Omega\times\mathbb{R}^{+}\times\mathbb{R}$.
\item[(2)]
$\widetilde{F}(x,t,\eta)\geq \frac{1}{\mu}\left(\frac{1}{2}\eta^2-\eta g(x)+\eta\Delta g(x)\right)$
for every $(x,t,\eta) \in \Omega\times\mathbb{R}^{+}\times\mathbb{R}$.
\item[(3)]
There exist positive constants $a$ and $b<2$ such that
$$\widetilde{F}(x,t,\eta)\leq a(1+|\eta|^{b})+\frac{1}{\mu}\left(\frac{1}{2}\eta^2-\eta g(x)+\eta\Delta g(x)\right).$$
\item[(4)]
There exist positive constants $\alpha$, $\beta$ with $\beta>\alpha\kappa$ such that
$$\frac{\inf_{x \in \Omega}\left(\widetilde{F}(x,t,\beta)-\frac{1}{\mu}\left(\frac{1}{2}\beta^2-\beta g(x)+\beta\Delta g(x)\right)\right)}{{\beta}^{2}}>m_1\frac{K_1}{\alpha}+m_2K_2{\alpha}^{q-2},$$
where $\kappa,K_1,K_2$ are given by \eqref{p8}.
\end{itemize}
Then the problem \eqref{In1} has at least three weak solutions in $H_{0}^{1}(\Omega)$, for each parameter $\mu$  belonging to
$\Lambda(\alpha,\beta):= \frac{2(2^N-1)}{D^2}\times \left(\delta_1,\delta_2\right)$,
where\\ $\delta_1:= \frac{{\beta}^{2}}{\inf_{x \in \Omega}\left(\widetilde{F}(x,t,\beta)-\frac{1}{\mu}\left(\frac{1}{2}\beta^2-\beta g(x)+\beta\Delta g(x)\right)\right)}$ and $\delta_2:=\frac{1}{m_1\frac{K_1}{\alpha}+m_2K_2\alpha^{q-2}}$.
\end{theorem}
\begin{proof}
Set $X:=H_{0}^{1}(\Omega)$ and define the functionals $\varphi(u)$ and
$\vartheta(u)$ by Definition \ref{functional}. Clearly, $\vartheta$ and $\varphi$ satisfy the assumptions of \cite[Theorem 3.6]{bonanno2010structure}.
By (1)
\begin{equation}\label{m1}
\widetilde{F}(x,t,\eta)\leq \frac{1}{\mu}\left(\frac{1}{2}\eta^{2}-\eta g(x)+\eta\Delta g(x)\right)+m_1|\eta|+m_2\frac{|\eta|^{q}}{q}
\end{equation}
for every
$(x,t,\eta)\in\Omega\times\mathbb{R}^{+}\times\mathbb{R}$. Thus
\[
\begin{array}{rl}
\vartheta(u):=& \int_{\Omega}\widetilde{F}(x,t,u)dx-\frac{1}{2\mu}\int_{\Omega}\left(u(x,t)\right)^2dx+\frac{1}{\mu}\int_{\Omega}g(x)u(x,t)dx\\& \\
&\quad -\frac{1}{\mu}\int_{\Omega}\Delta g(x)u(x,t)dx\\ & \\
\leq& \frac{1}{\mu}\int_{\Omega}\left(\frac{1}{2}(u(x,t))^{2}-u(x,t) g(x)+u(x,t)\Delta g(x)\right)dx \\ & \\
& \quad +\int_{\Omega}\left(m_1|u(x,t)|+m_2\frac{|u(x,t)|^{q}}{q}\right)dx-\frac{1}{2\mu}\int_{\Omega}\left(u(x,t)\right)^2dx
\\ & \\
&\quad +\frac{1}{\mu}\int_{\Omega}g(x)u(x,t)dx-\frac{1}{\mu}\int_{\Omega}\Delta g(x)u(x,t)dx\\ & \\
\leq& m_1\parallel u\parallel_{L^1(\Omega)}+\frac{m_2}{q}\parallel u\parallel_{L^q(\Omega)}^q.
\end{array}
\]
Let
$r \in ]0,+\infty[$ such that $ \varphi(u)\leq r$. By \eqref{p6},
\[
\vartheta(u)\leq\left(\sqrt{2r}c_1m_1+\frac{2^{\frac{q}{2}}c_q^qm_2}{q}r^{\frac{q}{2}}\right).
\]
Set $\chi(r):=\frac{\sup_{u\in\varphi^{-1}]-\infty,r[}\vartheta(u)}{r}$.
Consequently
\begin{equation}\label{m3}
\chi(r)\leq\left(\sqrt{\frac{2}{r}}c_1m_1+\frac{2^{\frac{q}{2}}c_q^qm_2}{q}r^{\frac{q}{2}-1}\right),
\end{equation}
for every $r>0$.

By \eqref{p8}, there is $x_0\in\Omega$ such that $B(x_0,D)\subseteq\Omega$. Set
\begin{eqnarray}
u_\beta(x,t):=\left\{
  \begin{array}{ll}
  0 \ \ \ \ \  x \in\Omega\backslash B(x_0,D), \\
  \frac{2\beta}{D}(D-|x-x_0|) \ \ x\in\ B(x_0,D)\backslash B(x_0,D/2),\\
\beta  \ \ \ \ \  x\in B(x_0,D/2).
  \end{array}
\right.
\end{eqnarray}
Thus
$u_{\beta}\in H_0^1(\Omega)$.
So
\begin{eqnarray}\label{m5}
\nonumber \varphi(u_\beta)&=&\frac{1}{2}\int_{\Omega}|\nabla u_\beta(x,t) |^2dx \\
\nonumber&=&\frac{1}{2}\int_{B(x_0,D)\backslash B(x_0,D/2)}\frac{(2\beta)^2}{D^2}dx \\
\nonumber&=&\frac{1}{2}\frac{(2\beta)^2}{D^2}(meas(B(x_0,D))-meas(B(x_0,D/2)))\\
&=&\frac{1}{2}\frac{(2\beta)^2}{D^2}\frac{\pi^{N/2}}{\Gamma(N/2+1)}\left(D^N-(D/2)^N\right).
\end{eqnarray}
If we force
$\beta>\alpha\kappa$, by (4), $\alpha^2<\varphi(u_{\beta})$ because $\alpha^2<\frac{\beta^2}{\kappa^2}$. Also by assumption (2),
\begin{equation}\label{m6}
\begin{array}{rl}
\vartheta(u_\beta):=&\int_{\Omega}\widetilde{F}(x,t,u_\beta)dx-
\frac{1}{2\mu}\int_{\Omega}\left(u_\beta(x,t)\right)^2dx+\frac{1}{\mu}\int_{\Omega}g(x)u_\beta(x,t)dx\\ & \\
&\quad \quad -\frac{1}{\mu}\int_{\Omega}\Delta g(x)u_\beta(x,t)dx\\ & \\ =&\int_{\Omega}\left[\widetilde{F}(x,t,u_\beta)-\frac{1}{\mu}\left(\frac{1}{2}u_\beta(x,t)^2-g(x)u_\beta(x,t)+\Delta g(x)u_\beta(x,t)\right)\right]dx\\ & \\
\geq&\int_{B(x_0,D/2)}\left[\widetilde{F}(x,t,u_\beta)-\frac{1}{\mu}\left(\frac{1}{2}u_\beta(x,t)^2-g(x)u_\beta(x,t)+\Delta g(x)u_\beta(x,t)\right)\right]dx\\ & \\
\geq &\inf_{x\in\Omega}\left(\widetilde{F}(x,t,\beta)-\frac{1}{\mu}\left(\frac{1}{2}\beta^2-\beta g(x)+\beta\Delta g(x)\right)\right)\frac{\pi^{N/2}}{\Gamma(N/2+1)}\frac{D^N}{2^N}.
\end{array}
\end{equation}
Next by dividing \eqref{m5} on \eqref{m6}, we have
\begin{equation}\label{m7}
\frac{\vartheta(u_\beta)}{\varphi(u_\beta)}\geq \frac{D^2}{2(2^N-1)}\frac{\inf_{x\in\Omega}\left(\widetilde{F}(x,t,\beta)-\frac{1}{\mu}\left(\frac{1}{2}\beta^2-\beta g(x)+\beta\Delta g(x)\right)\right)}{\beta^2}.
\end{equation}
Using \eqref{m3}, assumption (4) implies
\begin{eqnarray}\label{m8}
\nonumber\chi(\alpha^2)&\leq&(\frac{\sqrt{2}c_1m_1}{\alpha}+\frac{2^{\frac{q}{2}}c_q^q m_2\alpha^{q-2}}{q})\\
\nonumber&=&\frac{D^2}{2(2^N-1)}(m_1\frac{K_1}{\alpha}+m_2K_2{\alpha}^{q-2})\\
\nonumber&<&\frac{D^2}{2(2^N-1)}\frac{\inf_{x\in\Omega}(\widetilde{F}(x,t,\beta)-U(x,t)-G(x)-\overset{\Delta}G(x))}{\beta^2}\\
\nonumber&\leq&\frac{\vartheta(u_\beta)}{\varphi(u_\beta)}.
\end{eqnarray}
Assuming $b<2$ and considering $|u|^b\in L^{\frac{2}{s}}(\Omega)$ for all $u \in X$,
 H\"{o}lder's inequality for $ u \in X$ implies
$\int_{\Omega}|u(x,t)|^b dx\leq\parallel u\parallel_{L^2(\Omega)}^b (meas(\Omega))^{\frac{2-b}{2}}$.
Therefore equation \eqref{p6} shows for all $u \in X$
\[
\int_{\Omega}|u(x,t)|^b dx\leq c_2^b\parallel u\parallel^b (meas(\Omega))^{\frac{2-b}{2}},
\]
and by assumption (3),
\[
\begin{array}{rl}
I_\mu(u)=& \varphi(u)-\mu\vartheta(u)\\ & \\
=& \frac{\parallel u\parallel^2}{2}-\mu\int_{\Omega}\widetilde{F}(x,t,u)dx+
\frac{1}{2}\int_{\Omega}\left(u(x,t)\right)^2dx\\&\\
&\quad -\int_{\Omega}g(x)u(x,t)dx+\int_{\Omega}\Delta g(x)u(x,t)dx\\&\\
\geq & \frac{\parallel u\parallel^2}{2}-\mu\int_{\Omega}a\left(1+|u(x,t)|^{b}\right)dx\\ &\\
\geq &\frac{\parallel u\parallel^2}{2}-\mu a c_2^b (meas(\Omega))^{\frac{2-b}{2}}\parallel u\parallel^b-a\mu meas(\Omega).
\end{array}
\]
This means for every $
\mu \in \Lambda(\alpha,\beta)\subseteq \left] \frac{\vartheta(u_\beta)}{\varphi(u_\beta)},\frac{\alpha^2}{\sup_{\varphi(u)\leq \alpha^2}\vartheta(u)}\right[$,
$I_\mu$ is coercive.
Therefore by Theorem \ref{p1} for each
$\mu \in \Lambda(\alpha,\beta)$
the functional $I_\mu$
has at least three distinct critical points that they are weak solutions of the problem
\eqref{In1}.
\end{proof}
%----------------------------------------------------------------------------------------------------------------------------------------------------------
\section{Numerical Experiment}
Now, we present an example.

\begin{equation}\label{N1}
\left\{
\begin{array}{l}
  \frac{\partial u}{\partial t}-\frac{\partial(\Delta u)}{\partial t}=\frac{1}{100}\frac{99}{100t}\left(1+\frac{\exp(-t)}{99}\right)(8+100u+u^2) \ \in \Omega , %\label{N1}\\
  u\mid_{\partial\Omega}=0,\\ %\label{N2}\\
  u(x,0)=\frac{1}{1000}\left(\frac{1}{100}-\left(x_1^2+x_2^2+x_3^2\right)\right) \ \ x\in\Omega,
\end{array}
\right.
\end{equation}
where $\Omega :=\left\{(x_1,x_2,x_3)\in \mathbb{R}^3, x_1^2+x_2^2+x_3^2\leq0.1\right\}$, then $\mu=0.01, N=3, D=r=0.1, 2^*=6,$ $g(x)=0.001\left(0.01-\left(x_1^2+x_2^2+x_3^2\right)\right),$ $\Delta g(x)=-0.006$ and
$f(x,t,u)=\frac{99}{100t}\left(1+\frac{\exp(-t)}{99}\right)(8+100u+u^2)$. Now, setting $q=3$, then
\[
\begin{array}{l}
c_1\leq0.00445759,\quad  c_q\leq0.171543,\\
\kappa=1.16798,\quad  K_1\leq8.82557,\quad  K_2\leq6.66307.
\end{array}
\]
Clearly $F(x,t,s)=\frac{99}{100}\left(1+\frac{\exp(-t)}{99}\right)\left(8+100s+s^2\right)$, suppose
$m_1=9$ and $m_2=1$, then the assumption (1) of the Theorem \ref{Maint} is satisfied, i.e.
\[
\begin{array}{l}
\frac{99}{100}\left(1+\frac{\exp(-t)}{99}\right)\left(8+100s+s^2\right)\leq \\
9+ s^{2}+\frac{1}{0.01}\left(s-0.001\left(0.01-\left(x_1^2+x_2^2+x_3^2\right)\right)-0.006\right),
\end{array}
\]
for all $(x,t,s)\in\Omega\times\mathbb{R}^{+}\times\mathbb{R}$.\\ Obviously $\widetilde{F}(x,t,\eta)=\frac{99}{100}\left(1+\frac{\exp(-t)}{99}\right)\left(8\eta+50\eta^2+\frac{\eta^3}{3}\right)$,
then it can be easily verified that the assumption (2) of the Theorem \ref{Maint} holds, i.e.
for all $(x,t,s)\in\Omega\times\mathbb{R}^{+}\times\mathbb{R}$
\[
\begin{array}{l}
\frac{99}{100}\left(1+\frac{\exp(-t)}{99}\right)
\left(8\eta+50\eta^2+\frac{\eta^3}{3}\right)\geq \\
\frac{1}{0.01}\left(\frac{1}{2}\eta^2- 0.001\eta\left(0.01-\left(x_1^2+x_2^2+x_3^2\right)\right)-0.006\eta\right).
\end{array}
\]
Also, by choosing $a=b=10$, the assumption (3) of the Theorem \ref{Maint} is satisfied, i.e.
for all $(x,t,s)\in\Omega\times\mathbb{R}^{+}\times\mathbb{R}$
\[
\begin{array}{l}
\frac{99}{100}\left(1+\frac{\exp(-t)}{99}\right)\left(8\eta+50\eta^2+\frac{\eta^3}{3}\right)\leq\\
10(1+\eta^{10})+\frac{1}{0.01}\left(\frac{1}{2}\eta^2- 0.001\eta\left(0.01-\left(x_1^2+x_2^2+x_3^2\right)\right)-0.006\eta\right).
\end{array}
\]
More, set $\alpha=1$ and $\beta=500>\alpha\kappa$ hence, for all $t\geq 0$, it is not difficult to see that
\begin{eqnarray}
\nonumber
162.872&=&\frac{\inf_{x \in \Omega}\left\{\left(
\begin{array}{c}
\frac{99}{100}\left(1+\frac{\exp(-t)}{99}\right)\left(8\eta+50\eta^2+\frac{\eta^3}{3}\right)- \\
\frac{1}{0.01}\left(\frac{1}{2}\eta^2- 0.001\eta\left(0.01-\left(x_1^2+x_2^2+x_3^2\right)\right)-0.006\eta\right) \\
\end{array}
\right)
\right\}}{{\beta}^{2}}\\ \nonumber&>&m_1K_1+m_2K_2=86.0932.
\end{eqnarray}
Furthermore, it is observed that
$\mu=0.01\in\left]\frac{1}{162.872}%=0.0060784,0.0116153=
,\frac{1}{86.0932}\right[$,
therefore the problem \eqref{N1} admits at least three week solutions in according to the Theorem \ref{Maint}.
%*********************************Bibliography******************************************************

\textbf{Acknowledgements}\\
The authors are very grateful to anonymous reviewers for carefully reading the paper and for their comments and
suggestions which have improved the paper very much.
%\section*{References}
\bibliographystyle{abbrv}
\bibliography{myref}
%-------------------------------------------------------------------------------------------------------
\end{document}